\documentclass[11pt,a4paper]{article}
\usepackage[top=3cm, bottom=3cm, left=3cm, right=3cm]{geometry}
\usepackage{tabu}
\usepackage{breqn}
\usepackage[hang,flushmargin]{footmisc} 

\usepackage{amsmath,amsthm,amssymb,graphicx, multicol, array}
\usepackage{enumerate}
\usepackage{enumitem}
\usepackage{hyperref}
\usepackage{setspace}
\usepackage{xcolor}
\usepackage{currfile}

\usepackage{tabto}

\DeclareMathOperator{\tdiv}{div}

\newtheorem{thm}{Theorem}[section]
\newtheorem{lem}{Lemma}[section]
\newtheorem{conj}{Conjecture}[section]

\newtheorem{cor}{Corollary}[section]

\newcommand{\N}{\mathbb{N}}
\newcommand{\Z}{\mathbb{Z}}

\usepackage{fancyhdr}
\title{Prime Values of Quadratic Polynomials}
\date{}
\author{N. A. Carella}

\begin{document}
\thispagestyle{empty}
\date{}

\maketitle
\textbf{\textit{Abstract}:} This note shows that the product $e \pi$ of the natural base $e$ and the circle number $\pi$ is an irrational number. \let\thefootnote\relax\footnote{ \today \date{} 
}

\tableofcontents
\section{Introduction} \label{s2579}
The basic problem of prime values of linear polynomials $f(t)=qt+a\in \Z[t]$ is completely solved. Dirichlet theorem for primes in arithmetic progressions proves that any admissible linear polynomial has infinitely many prime values. The quantitative form of this theorem has the asymptotic formula
\begin{equation}\label{eq2579.403}
\sum_{\substack{n\leq x\\n\equiv a \bmod q}}\Lambda(n)\sim \frac{1}{\varphi(q)}x,
\end{equation}
where $ \gcd(a,q)=1 $, as $ x\to \infty $. The next basic problem of prime values of quadratic polynomials $f(t)=at^2+bt+c\in \Z[t]$ has very precise heuristics and many partial results, but there is no qualitative nor quantitative results known. This note investigates the prime values of the admissible quadratic polynomials $f(t)=qt^2+a\in \Z[t]$, and proposes the following result.

\begin{thm} \label{thm2579.505} Let $x\geq 1$ be a large number. Let $a$ and $q$ be a pair of relatively prime integers, with opposite parity, and $q\ll (\log x)^b$, where $ b\geq 0 $ is a constant. Then,
\begin{equation}\label{eq2579.407}
\sum_{\substack{n\leq x^{1/2}\\n \text{ odd}}}\Lambda(qn^2+a)\gg \frac{q}{2\varphi(q)}x^{1/2}+O\left ( x^{1/2}e^{-c\sqrt{\log x}} \right)\nonumber,
\end{equation}
where $  c>0$ is an absolute constant.
\end{thm}
The core of the proof in Section \ref{s2500} consists of the \textit{quadratic to linear identity} in Section \ref{2595}, and other results proved in Section \ref{s2585} to Section \ref{s2905}. Theorem \ref{thm2579.505} proves the predicted asymptotic formula  
\begin{equation}\label{eq2579.409}
\sum_{\substack{n\leq x^{1/2}\\n \text{ odd}}}\Lambda(qn^2+a)\sim \frac{c_f}{2}x^{1/2},
\end{equation}
but not the constant $c_f\geq 0$, see \cite{BH62} for finer details. The conjectured constant, which depends on the  polynomial $f(t)=qt^2+a$, has the form
\begin{equation}\label{eq2579.410}
c_f=\epsilon \prod_{\substack{p\geq 3\\p\,\mid \,q}}\left ( \frac{p}{p-1}\right )\prod_{\substack{p\geq 3\\p\,\nmid \,q}}\left (1-\left ( \frac{-aq}{p}\right ) \frac{1}{p}\right ),
\end{equation}
where
\begin{equation}\label{eq2579.412}
\epsilon =\left \{\begin{array}{ll}
1/2 & \text{ if } q\not\equiv 0 \bmod 2,  \\
1 & \text{ if } q \equiv 0 \bmod 2. \\
\end{array} \right .
\end{equation}
The conjectured general formula for the constant $c_f(a,c)\geq 0$ attached to an admissible quadratic polynomial $ f(t)=at^2+bt+c \in \Z[t]$ appears 
in \cite[p.\ 46]{HL23}, \cite[p.\ 364]{BH62}, et alii. Discussions on the convergence of the product \eqref{eq2579.410} appears in \cite{BH62}, \cite[Section 5]{FG18}, et alii. Results on the average value $\overline{c_f(a,c)}$, and other properties appear in \cite{BZ07}, \cite{RI15}, et cetera, optimization and numerical techniques appear in \cite{JW03}, and similar references. \\

The result in Theorem \ref{thm2579.505} is a special case of the Bateman-Horn Conjecture for polynomials over the integers, see \cite{BH62}, and \cite{FG18} 
for a survey. Other results and  recent discussions are given in \cite[p. 406]{RP96}, \cite[p.\ 342]{NW00}, et cetera. Some partial results are proved in \cite{GM00}, \cite{MA09}, \cite{BZ07}, \cite{DI82}, \cite{IH78}, 
\cite{LR12}, and the recent literature.

\section{Ramanujan Sums} \label{s2585}
For a pair of integers $q\geq 1$, and $m\ne0$, the simple exponential sum 
\begin{equation}\label{eq2585.800}
c_q(m)=\sum_{1\leq a < q, \;\gcd(a,q)=1}=e^{i2\pi am/q}
\end{equation}
is known as the Ramanujan sum. Some of the basic properties of the function $c_{q}(m)$ are listed here.

\begin{lem} \label{lem2585.810}  Let $m$ and $q$ be integers, and let $\mu$ be the Mobius function. Then 
\begin{enumerate}[font=\normalfont, label=(\roman*)]
\item $c_{q}(m)=\mu(q/d) \varphi(q)/\varphi(q/d)$, where $d = \gcd(m, q)$.
\item $c_{q}(m)=\sum_{d \mid \gcd(m,q)} \mu(q/d)d$.
\end{enumerate}
\end{lem}
The proofs of these basic properties of $c_{q}(m)$ are discussed in \cite[Theorem 8.6]{AP76}, \cite[Theorem 4.1]{MV07}, et alii. 

\begin{lem}\label{lem2585.824} Given a large number $ x \geq 1 $, let $p\geq 2$ be a large prime such that $x< p$, and let $ N=2p $. If $n\leq x$ is a fixed odd integer, and $s\leq x^{1/2}$, then,
\begin{enumerate}[font=\normalfont, label=(\roman*)]
\item $\displaystyle c_N(s-n)=\sum_{\substack{1\leq u<N\\\gcd(u,N)=1\\ s-n\ne 0}}e^{i2\pi u(s-n)/N}=\left (-1\right )^s.$
\item $\displaystyle c_N(s^2-n)=\sum_{\substack{1\leq u<N\\\gcd(u,N)=1\\s^2-n\ne 0}}e^{i2\pi u(s^2-n)/N}=\left (-1\right )^s.$
\end{enumerate}
\end{lem}
\begin{proof}(i) For fixed odd integer $n\leq x$, and $s\leq x^{1/2}$, where $ s-n\ne0 $, the absolute value of the difference of these integers 
satisfies $0<|s-n|<p$. Hence, for $N=2p$, 
\begin{equation}\label{eq2585.802}
\gcd(s-n,2p)=\left \{\begin{array}{ll}
1 & \text{ if } s\equiv 0 \bmod 2,  \\
2 & \text{ if } s\equiv 1 \bmod 2. \\
\end{array} \right .
\end{equation} 
By Lemma \ref{lem2585.810}, at $q=N$ and $m=s-n$, the Ramanujan sum has the value
\begin{equation}\label{eq2585.804}
c_N(s-n)=\sum_{d\mid \gcd(s-n,N)}d\mu(N/d)=\left \{\begin{array}{ll}
1 & \text{ if } \gcd(s-n),2p)=1,  \\
-1 & \text{ if } \gcd(s-n),2p)=2. \\
\end{array}  \right .
\end{equation}
Combining \eqref{eq2585.802} and \eqref{eq2585.804} yield $ c_N(s-n)=(-1)^s $. \\

(ii) The same proof applies to this case.
\end{proof}

\begin{lem}\label{lem2585.806} Given a large number $ x \geq 1 $, let $p\geq 2$ be a large prime such that $x< p$, and let $ N=2p $. Further, assume that $[ x^{1/2} ]=2k$ is an even integer. If $n\leq x$ is an odd integer, then,
\begin{enumerate} [font=\normalfont, label=(\roman*)]
\item $\displaystyle \sum_{\substack{s\leq x^{1/2}\\s-n\ne 0}}c_N(s-n)=\left \{\begin{array}{ll}
0 & \text{ if } [x^{1/2}]=2k,  \\
-1 & \text{ if } [x^{1/2}]=2k\pm 1. \\
\end{array}  \right .$
\item $\displaystyle \sum_{\substack{s\leq x^{1/2}\\s^2-n\ne 0}}c_N(s^2-n)=\left \{\begin{array}{ll}
0 & \text{ if } [x^{1/2}]=2k,  \\
-1 & \text{ if } [x^{1/2}]=2k\pm 1. \\
\end{array}  \right .$
\end{enumerate}
\end{lem}
\begin{proof}(i) For a fixed odd integer $n\ne s^2\leq x$, and $s\leq x^{1/2}$ such that $ s-n\ne0 $, the finite sum $c_N(s-n)=(-1)^s$, see Lemma \ref{lem2585.824}. Thus, 
\begin{equation}\label{eq2585.902}
\sum_{\substack{s\leq x^{1/2}\\s-n\ne 0}}c_N(s-n)=\sum_{s\leq x^{1/2}}(-1)^s=\left \{\begin{array}{ll}
0 & \text{ if } [x^{1/2}]=2k,  \\
-1 & \text{ if } [x^{1/2}]=2k\pm 1. \\
\end{array}  \right .
\end{equation} 
(ii) The same proof applies to this case.
\end{proof}

\section{Characteristic Functions For Integer Powers} \label{s2535}
An explicit representation of the characteristic function $\mathcal{Q}:\N\longrightarrow \{0,1\}  $ of square odd integers on an interval $[1,x]$ is introduced below. The parameters were chosen to fit the application within.
\begin{lem} \label{lem2535.103}
Given a large number $ x \geq 1 $, let $p\geq 2$ be a large prime such that $x< p$, and let $ N=2p $. Further, assume that $[ x^{1/2} ]=2k$ is an even integer. If $n \leq x$ is a fixed odd integer, then, 
\begin{equation}\label{eq2535.103}
\mathcal{Q}(n)= \frac{1}{\varphi(N)}\sum _{1\leq s\leq x^{1/2},}\sum _{\substack{0\leq u\leq N-1\\\gcd(u,N)=1}} e^{i2\pi \left (s ^{2}-n\right)u/N}
=\left \{\begin{array}{ll}
1 & \text{ if } n=s^2,  \\
\displaystyle 0 & \text{ if } n\ne s^2. \\
\end{array} \right .\nonumber
\end{equation}
\end{lem}

\begin{proof} \textit{Assume $ n\ne s^2 $ is not a square}. The hypothesis $N=2p$, with $p>2$ prime, and an odd integer $n\leq x$ and $s\leq x^{1/2}$, imply that the Ramanujan sum has the value $ c_N(s^2-n)=(-1)^s $ for $s^2-n\ne0  $, see Lemma \ref{lem2585.824}. Hence, 
\begin{eqnarray}\label{eq2535.107}
\frac{1}{\varphi(N)}\sum _{1\leq s\leq x^{1/2},}\sum _{\substack{0\leq u\leq N-1\\\gcd(u,N)=1\\n\ne s^2}} e^{i2\pi \left (s ^{2}-n\right)u/N}
&=&\frac{1}{\varphi(N)}\sum _{\substack{1\leq s\leq x^{1/2}\\n\ne s^2}}c_N(s^2-n)\nonumber\\
&=&0,
\end{eqnarray}
the last equality follows from Lemma \ref{lem2585.806} since $ [ x^{1/2} ]=2k$ is an even integer.\\

\textit{Assume $ n=s^2 $ is a square}. The hypothesis $n\leq x$ implies that the equation $s ^{2}-n=0$ has a unique integer solution $s\in [1,x^{1/2}]$ for each square integer $n\in [1, x]$. Thus, the double finite sum collapses to 
\begin{equation} \label{eq2535.105}
\frac{1}{\varphi(N)}\sum _{1\leq s\leq x^{1/2},}\sum _{\substack{0\leq u\leq N-1\\\gcd(u,N)=1}} e^{i2\pi \left (s ^{2}-n\right)u/N}
=1 
\end{equation}
for exactly one value $s=n^2$. The result follows from these observations.
\end{proof}
This technique is very flexible, and has the advantages of being easily extended to other classes of integer powers as cubic integers, and quartic integers, et cetera. 

\section{Quadratic To Linear Identity} \label{2595}
The quadratic to linear inequality trades off the evaluation of
\begin{equation}
 \sum_{ n\leq x^{1/2},\text{odd } n} \Lambda(qn^2+a)    
\end{equation}
for the evaluation of a product of some exponential sums and
\begin{equation}
 \sum_{ n\leq x,\text{odd } n} \Lambda(qn+a)   
\end{equation}
\begin{lem} \label{lem2595.001} Given a large number $ x \geq 1 $, let $p\geq 2$ be a large prime such that $x< p$, and let $ N=2p $. Further, assume that $[ x^{1/2} ]=2k$ is an even integer. If $a$ and $q$ is a pair of relatively prime integers, and opposite parity, then, \begin{equation}\label{eq2595.006}
\sum_{\substack{ n\leq x^{1/2}\\ \text{odd } n}} \Lambda(qn^2+a)=\frac{1}{\varphi(N)}\sum_{\substack{n\leq x\\n \text{ odd}}}\Lambda(qn+a)\sum_{s\leq x^{1/2},}\sum_{\substack{0\leq u<N\\\gcd(u,N)}}e^{i2\pi (s^2-n)u/N}.\nonumber
\end{equation}
\end{lem}

\begin{proof} Summing the product of $ \Lambda(qn+a) $ and the characteristic function $\mathcal{Q}(n)$ over the odd integers $ n\leq x$ returns
\begin{eqnarray}\label{eq2595.300}
\sum_{\substack{n\leq x\\n \text{ odd}}}\Lambda(qn+a)\mathcal{Q}(n)&=&\sum_{\substack{n\leq x\\n \text{ odd}}}\Lambda(qn+a) \frac{1}{\varphi(N)}\sum_{s\leq x^{1/2},}\sum_{\substack{0\leq u<N\\\gcd(u,N)}}e^{i2\pi (s^2-n)u/N}\nonumber\\
&=&\sum_{\substack{ m\leq x^{1/2}\\ \text{odd } m}} \Lambda(qm^2+a),
\end{eqnarray} 
where $ m^2=n\leq x $. The last line follows from the definition of $ \mathcal{Q}(n) $, see Lemma \ref{lem2535.103}. 
\end{proof}

This concept has a straight forward extension to the \textit{cubic to linear identity}, the \textit{quartic to linear identity}, et cetera.

\section{The Main Term} \label{s2505}
The main term is evaluated in this section.   
  
\begin{lem}\label{lem2505.101} Given a large number $ x \geq 1 $, let $p \geq 2$ be a large prime such that $x< p$, and let $ N=2p $. Further, assume that $[ x^{1/2} ]=2k$ is an even integer. Let $a$ and $q$ be a pair of relatively prime integers, with opposite parity, and $q\ll (\log x)^b$, where $ b\geq 0 $ is a constant. Then,
\begin{equation}\label{eq2505.101}
\frac{1}{\varphi(N)}\sum_{\substack{n\leq x\\n \text{ odd}}}\Lambda(qn+a)\sum_{\substack{1\leq u<N\\\gcd(u,N)=1}}\sum_{1\leq s\leq x^{1/2}}e^{i2\pi u(s-n)/N}\gg\frac{q}{2\varphi(q)}x^{1/2}+O\left ( x^{1/2}e^{-c\sqrt{\log x}} \right)\nonumber,
\end{equation}
where $c>0$ is a constant.
\end{lem}
\begin{proof} Consider the dyadic partition
\begin{eqnarray}\label{eq2505.103}
M(x)&=&\frac{1}{\varphi(N)}\sum_{\substack{n\leq x\\n \text{ odd}}}\Lambda(qn+a)\sum_{\substack{1\leq u<N\\ \gcd(u,N)=1}}\sum_{1\leq s\leq x^{1/2}}e^{i2\pi u(s-n)/N}\\
&=&\frac{1}{\varphi(N)}\sum_{\substack{n\leq x\\n \text{ odd}\\n=s}}\Lambda(qn+a)\sum_{1\leq s\leq x^{1/2}}\sum_{\substack{1\leq u<N\\ \gcd(u,N)=1}}e^{i2\pi u(s-n)/N}\nonumber\\
&&\qquad +\frac{1}{\varphi(N)}\sum_{\substack{n\leq x\\n \text{ odd} \\n\ne s}}\Lambda(qn+a)\sum_{1\leq s\leq x^{1/2}}\sum_{\substack{1\leq u<N\\ \gcd(u,N)=1}}e^{i2\pi u(s-n)/N}\nonumber \\
&=&M_0(x)+M_1(x)\nonumber.
\end{eqnarray}
\textit{The Subsum $M_0(x)$}. The first term in \eqref{eq2505.103} has the value:
\begin{eqnarray}\label{eq2505.105}
M_0(x)&=&\frac{1}{\varphi(N)}\sum_{\substack{n\leq x\\n \text{ odd}\\n=s}}\Lambda(qn+a)\sum_{1\leq s\leq x^{1/2}}\sum_{\substack{1\leq u<N\\ \gcd(u,N)=1}}e^{i2\pi u(s-n)/N}\\
&=&\sum_{\substack{n\leq x^{1/2}\\n \text{ odd}}}\Lambda(qn+a)\nonumber\\
&\gg&\frac{q}{2\varphi(q)}x^{1/2}+O\left ( x^{1/2}e^{-c\sqrt{\log x}} \right)\nonumber,
\end{eqnarray} 
since $ qn+a\leq qx^{1/2}+a $, as $ x\to \infty $, see \cite[Theorem 8.8]{EL85}, \cite[Corollary 11.19]{MV07}, et cetera.
\\

\textit{The Subsum $M_1(x)$}. The second term in \eqref{eq2505.103} has the value
\begin{eqnarray}\label{eq2505.107}
M_1(x)&=&\frac{1}{\varphi(N)}\sum_{\substack{\substack{n\leq x\\n \text{ odd}} \\n\ne s}}\Lambda(qn+a)\sum_{1\leq s\leq x^{1/2}}\sum_{\substack{1\leq u<N\\\gcd(u,N)=1}}e^{i2\pi u(s-n)/N}\\
&=&\frac{1}{\varphi(N)}\sum_{\substack{\substack{n\leq x\\n \text{ odd}} \\n\ne s}}\Lambda(qn+a)\sum_{1\leq s\leq x^{1/2}}c_N(s-n)\nonumber\\
&=&0\nonumber,
\end{eqnarray}
where $ c_N(s-n) $ is a Ramanujan sum. The last equality in \eqref{eq2505.107} follows from Lemma \ref{lem2585.806} since $ n\leq x $ is odd, and $ [ x^{1/2} ]=2k$. The sum $M(x)=M_0(x)+M_1(x)$ completes the verification.
\end{proof}

\section{The Error Term} \label{s2905}
The error term is evaluated in this section. 
\begin{lem}\label{lem2905.002} Given a large number $ x \geq 1 $, let $p\geq2$ be a large prime such that $x< p$, and let $ N=2p\asymp xe^{c\sqrt{\log x}} $, where $c>0$ is a constant. Further, assume that $[ x ]=2k$ is an even integer. Then,
\begin{equation}\label{eq2905.002}
\frac{1}{\varphi(N)}\sum_{\substack{n\leq x\\n \text{ odd}}}\Lambda(qn+a)\sum_{1\leq s\leq x^{1/2}}\sum_{\substack{1\leq u<N\\\gcd(u,N)=1}}e^{i2\pi u(s-n)/N}\sum_{1<d\mid s}\lambda(d)=O\left ( x^{1/2}e^{-c\sqrt{\log x}} \right)\nonumber,
\end{equation}
where $c>0$ is a constant.
\end{lem}

\begin{proof} Switching the order of summation and a change of variable $s=dm$ return
\begin{eqnarray}\label{eq2905.120}  
E(x)&=&
\frac{1}{\varphi(N)}\sum_{\substack{n\leq x\\n \text{ odd}}}\Lambda(qn+a) \sum_{1\leq s\leq x^{1/2},}\sum_{\substack{1\leq u<N\\\gcd(u,N)=1}}e^{i2\pi u(s-n)/N}\sum_{1<d\mid s}\lambda(d)\nonumber\\
&=&
\frac{1}{\varphi(N)}\sum_{\substack{n\leq x\\n \text{ odd}}}\Lambda(qn+a) \sum_{1<d\leq x^{1/2}}\lambda(d)\sum_{\substack{1\leq s\leq x^{1/2}\\d\mid s}}\sum_{\substack{1\leq u<N\\\gcd(u,N)=1}}e^{i2\pi u(s-n)/N}\\
&=&
\frac{1}{\varphi(N)}\sum_{\substack{n\leq x\\n \text{ odd}}}\Lambda(qn+a) \sum_{1<d\leq x^{1/2}}\lambda(d)\sum_{1\leq m\leq x^{1/2}/d}\sum_{\substack{1\leq u<N\\\gcd(u,N)=1}}e^{i2\pi u(dm-n)/N}\nonumber.
\end{eqnarray}
A dyadic expansion returns
\begin{eqnarray}\label{eq2905.125}  
E(x)
&=&\frac{1}{\varphi(N)}\sum_{1<d\leq x^{1/2}}\lambda(d)\sum_{\substack{n\leq x\\n \text{ odd}\\n=dm}}\Lambda(qn+a) \sum_{1\leq m\leq x^{1/2}/d}\sum_{\substack{1\leq u<N\\\gcd(u,N)=1}}e^{i2\pi u(dm-n)/N}\nonumber\\
&&+
\frac{1}{\varphi(N)}\sum_{1<d\leq x^{1/2}}\lambda(d)\sum_{\substack{n\leq x\\n \text{ odd}\\n\ne dm}}\Lambda(qn+a) \sum_{1\leq m\leq x^{1/2}/d}\sum_{\substack{1\leq u<N\\\gcd(u,N)=1}}e^{i2\pi u(dm-n)/N}\\
&=&                                                                                                                                     E_0(x)+E_1(x)\nonumber.
\end{eqnarray}

\textit{The Subsum $E_0(x)$}. Set $n=dm$ and evaluate the finite sum: 
\begin{eqnarray}\label{eq2905.020}  
E_0(x)&=&
\frac{1}{\varphi(N)}\sum_{1<d\leq x^{1/2}}\lambda(d)\sum_{\substack{n\leq x\\n \text{ odd}\\n=dm}}\Lambda(qn+a) \sum_{1\leq m\leq x^{1/2}/d}\sum_{\substack{1\leq u<N\\\gcd(u,N)=1}}e^{i2\pi u(dm-n)/N}\nonumber\\
&=&                                                                                                                                                                                                                                                                                            \sum_{1<d\leq x^{1/2}}\lambda(d)\sum_{\substack{m\leq x^{1/2}/d\\dm \text{ odd}}}\Lambda(qdm+a)
\\
&=&O\left ( x^{1/2}e^{-c\sqrt{\log x}} \right)\nonumber,
\end{eqnarray}
where $c>0$ is a constant.\\

\textit{The Subsum $E_1(x)$}. Relabel the inner sum as a Ramanujan sum to obtain this: 
\begin{eqnarray}\label{eq2905.030}  
E_1(x)&=&
\frac{1}{\varphi(N)}\sum_{1<d\leq x^{1/2}}\lambda(d)\sum_{\substack{n\leq x\\n \text{ odd}\\n\ne dm}}\Lambda(qn+a) \sum_{1\leq m\leq x^{1/2}/d}\sum_{\substack{1\leq u<N\\\gcd(u,N)=1}}e^{i2\pi u(dm-n)/N}\nonumber\\
&=&\frac{1}{\varphi(N)}\sum_{1<d\leq x^{1/2}}\lambda(d)\sum_{\substack{n\leq x\\n \text{ odd}\\n\ne dm}}\Lambda(qn+a) \sum_{1\leq m\leq x^{1/2}/d}    c_N(dm-n).
\end{eqnarray}
Since $0<|dm-n|\leq x<p$, and $N=2p$, the expression $\gcd(dm-n,N)=1 \text{ or } 2$. Thus, the finite sum $c_N(dm-n)=\mu(N)=1$ or $c_N(dm-n)=\mu(N/2)=-1$, see Lemma \ref{lem2585.810}. Taking absolute value and using a trivial bounds, return 
\begin{eqnarray}\label{eq2905.038}  
\left | E_1(x)\right |
&\ll &\frac{1}{\varphi(N)}\left |\sum_{1<d\leq x^{1/2}}\lambda(d)\sum_{\substack{n\leq x\\n \text{ odd}}}\Lambda(qn+a) \sum_{1\leq m\leq x^{1/2}/d} \pm 1\right |\nonumber\\
&\ll &\frac{1}{\varphi(N)}\sum_{1<d\leq x^{1/2},}\sum_{1\leq m\leq x^{1/2}/d,} \sum_{\substack{n\leq x\\n \text{ odd}}}\Lambda(qn+a)\nonumber\\
&\ll &\frac{x}{\varphi(N)}\sum_{1<d\leq x^{1/2},}\sum_{1\leq m\leq x^{1/2}/d}1 \\
&\ll &\frac{x^{3/2}}{\varphi(N)}\sum_{1<d\leq x^{1/2}}\frac{1}{d} \nonumber\\
&=&O\left ( x^{1/2}e^{-c\sqrt{\log x}}\right )\nonumber,
\end{eqnarray}
where $\varphi(N)\gg xe^{c\sqrt{\log x}}/\log \log x$ for some constant $c>0$. The sum $E(x)=E_0(x)+E_1(x)$ completes the verification.
\end{proof}

\section{Prime Values of Quadratic Polynomials} \label{s2500}

For any pair of fixed integers $1\leq a\leq q$ such that $\gcd(a,q)=1$, the polynomial $f(t)=qt^2+a$ is irreducible, and it has fixed divisor $\tdiv(f)=\gcd(f(\mathbb{Z}))=1$, see \cite[p.\ 395]{FI10}, \cite[p.\ 46]{HL23}, \cite[p.\ 342]{NW00}, et cetera, for more details. \\

For an integer $n \geq 1$, the vonMangoldt function $\Lambda:\mathbb{N}\longrightarrow \mathbb{R}  $ is defined by 
\begin{equation}\label{eq2500.500}
\Lambda(n)= \left \{\begin{array}{ll}
\log p & \text{ if } n=p^k,  \\
0 & \text{ if } n\ne p^k, \\
\end{array} \right .
\end{equation}
where $n=p^k$ is a prime power, and the Euler totient function $\varphi:\mathbb{N}\longrightarrow \mathbb{Q}  $ is defined by $\varphi(n)=n\prod_{p\mid n}\left ( 1-1/p\right )$. A primes counting function, weighted by $ \Lambda(n) $, is defined by
\begin{equation}\label{eq2500.501}
\psi_2(x,q,a)=\sum_{\substack{n\leq x^{1/2}\\n \text{ odd}}}\Lambda(qn^2+a).
\end{equation}

\begin{proof}(Theorem \ref{thm2579.505}): Given a large number $ x \geq 1 $, let $p\geq 2$ be a large prime such that $x< p$, and let $ N=2p \asymp xe^{c\sqrt{\log x}} $, where $c>0$ is a constant. Further, assume that $[ x^{1/2}=2k ]$ is an even integer. Now, in terms of the \textit{quadratic to linear identity} in Lemma \ref{lem2595.001}, the weighted primes counting function has the form
\begin{eqnarray}\label{eq2500.504} 
\sum_{\substack{n\leq x^{1/2}\\n \text{ odd}}}\Lambda(qn^2+a)
&=&\frac{1}{\varphi(N)}\sum_{\substack{n\leq x\\n \text{ odd}}}\Lambda(qn+a)\sum_{s\leq x^{1/2},}\sum_{\substack{1\leq u<N\\\gcd(u,N)=1}}e^{i2\pi (s^2-n)u/N}\nonumber\\
&=&\frac{1}{\varphi(N)}\sum_{\substack{n\leq x\\n \text{ odd}}}\Lambda(qn+a)\sum_{\substack{1\leq u<N\\\gcd(u,N)=1}}e^{-i2\pi un/N}\sum_{1\leq s\leq x^{1/2}}e^{i2\pi us^2/N}.
\end{eqnarray}
This step removes any reference to nonlinear polynomial. Next step employs the square integers indicator function
\begin{equation}\label{eq2500.506}
 \sum_{d\mid n}\lambda(d)
=\left \{
\begin{array}{ll}
1 & \text{ if }  n \text{ is a square integer, }  \\
0 & \text{ if }  n \text{ is not a square integer, }
\end{array} \right .
\end{equation}
where $\lambda(n)$ is the Liouville function, see \cite[Theorem 2.19]{AP76}, to remove the nonlinear exponential term in the finite inner sum in \eqref{eq2500.504}. Furthermore, sampling the integers $s\leq x^{1/2}$, this procedure produces an effective lower bound  
\begin{eqnarray}\label{eq2500.508}  
\psi_2(x,q,a)&\geq&\frac{1}{\varphi(N)}\sum_{\substack{n\leq x\\n \text{ odd}}}\Lambda(qn+a)\sum_{\substack{1\leq u<N\\\gcd(u,N)=1}}e^{-i2\pi un/N}\nonumber \\
&& \hskip 1.50 in \times\sum_{1\leq s\leq x^{1/2}}\left ( 1+\sum_{1<d\mid s}\lambda(d) \right )e^{i2\pi us/N}\nonumber\\
&=&\frac{1}{\varphi(N)}\sum_{\substack{n\leq x\\n \text{ odd}}}\Lambda(qn+a)\sum_{\substack{1\leq u<N\\\gcd(u,N)=1}}\sum_{1\leq s\leq x^{1/2}}e^{i2\pi u(s-n)/N}\nonumber\\
&&+
\frac{1}{\varphi(N)}\sum_{\substack{n\leq x\\n \text{ odd}}}\Lambda(qn+a)\sum_{\substack{1\leq u<N\\\gcd(u,N)=1}}\sum_{1\leq s\leq x^{1/2}}e^{i2\pi u(s-n)/N}\sum_{1<d\mid s}\lambda(d)\\
&=&M(x)+E(x) \nonumber,
\end{eqnarray}
Applying Lemma \ref{lem2505.101} to the main term $M(x)$ and Lemma \ref{lem2905.002} to the error term $E(x)$ yield
\begin{eqnarray}\label{eq2500.516} 
\sum_{\substack{n\leq x^{1/2}\\n \text{ odd}}}\Lambda(qn^2+a)&\geq&M(x)+E(x)\nonumber\\
&\gg&\left [\frac{q}{2\varphi(q)}x^{1/2}+O\left ( x^{1/2}e^{-c\sqrt{\log x}} \right)\right ] +O\left ( x^{1/2}e^{-c\sqrt{\log x}} \right)  \\
&\gg&\frac{q}{2\varphi(q)}x^{1/2}+O\left ( x^{1/2}e^{-c\sqrt{\log x}} \right)\nonumber,
\end{eqnarray}
where $c> 0$ is an absolute constant, as $ x\to \infty $.
\end{proof}

\section{Euler Polynomial and Primes} \label{s9090}
The properties of the integers represented by the polynomial $f(t)=t^2+1\in \Z[t]$, such as squarefree values, almost prime values, and prime values, etc., are heavily studied in number theory. As early as 1760, Euler was developing the theory of prime values of polynomials. In fact, Euler computed an impressive large table of the prime values $p=n^2+1$, see \cite[p.\ 123]{EL00}. Probably, the prime values of polynomials was studied by other researchers before Euler. Later, circa 1910, Landau posed an updated question of the same problem about the primes values of this polynomial. A heuristic argument, based on circle methods, was demonstrated about two decades later. Surveys of the subsequent developments appear in \cite[p.\ 342]{NW00}, \cite[Section 19]{PJ09}, and similar references.

\begin{cor} \label{cor9090.19} Let $x \geq 1$ be a large number. Then
\begin{equation} \label{eq9090.19}
 \sum_{n\leq x^{1/2}} \Lambda(n^2+1)\gg \frac{x^{1/2}}{2}+O\left ( x^{1/2}e^{-c\sqrt{\log x}} \right),
\end{equation}
where $c>0$ is an absolute constant.
\end{cor}

\begin{proof} Consider the polynomial $ f(t)=4t^2+1\in \Z[t] $, where $ q=4 $ and $ a=1 $. Then, 
\begin{eqnarray} \label{eq9090.34}
 \sum_{n\leq x^{1/2}} \Lambda(n^2+1)&=& \sum_{n\leq x^{1/2}/2} \Lambda(4n^2+1)+O\left (\log x\right )
 \nonumber\\
 &\geq& \sum_{\substack{n\leq x^{1/2}/2\\n \text{ odd}}} \Lambda(4n^2+1)\\
&\gg& \frac{4}{2\varphi(4)}\frac{x^{1/2}}{2}+O\left ( x^{1/2}e^{-c\sqrt{\log x}} \right)
\nonumber,
\end{eqnarray}
where $c>0$ is an absolute constant.
The third line in \eqref{eq9090.34} follows from Theorem \ref{thm2579.505}.
\end{proof}
The standard heuristic for the prime values of the polynomial $ f(t)=t^2+1\in \Z[t] $ predicts the followings data.
\begin{conj} {\normalfont (\cite{HL23})} \label{conj909.12}Let $x \geq 1$ be a large number. Let $\Lambda$ be the vonMangoldt function, and let $\chi(n)=(n\mid p)$ be the quadratic symbol modulo $p$. Then
\begin{equation} \label{eq909.10}
\sum_{n\leq x^{1/2}} \Lambda \left( n^2+1 \right )
=c_fx^{1/2}+O\left (\frac{ x^{1/2}}{\log x }\right ),
\end{equation}
where the density constant
\begin{equation} \label{eq909.12}
c_f=\prod_{p \geq 3}\left( 1-\frac{\chi(-1)}{p-1} \right )=1.37281346 \ldots .
\end{equation}
\end{conj}

A list of the prime values of the polynomial $f(t)=t^2+1$ is archived in OEIS A002496.


\currfilename.\\


\begin{thebibliography}{998}
\bibitem{AP76} Apostol, Tom M. \textit{\color{blue}Introduction to analytic number theory}. Undergraduate Texts in Mathematics. Springer-Verlag, New York-Heidelberg, 1976.
\bibitem{BH62} Bateman, P. T., Horn, R. A. \textit{\color{blue}A heuristic asymptotic formula concerning the distribution of prime numbers.} Math. Comput., 16 (1962), pp. 363-367.
\bibitem{BZ07} Baier, Stephan; Zhao, Liangyi. \textit{\color{blue}Primes in quadratic progressions on average}. Math. Ann. 338 (2007), no. 4, 963-982. 
\bibitem{DI82} Deshouillers, J. M., Iwaniec, H. \textit{\color{blue}On the greatest prime factor of $n^2 + 1$}, Ann. Inst. Fourier (Grenoble) 32 (1982). 
\bibitem{EL85} Ellison, William; Ellison, Fern. \textit{\color{blue}Prime numbers}. A Wiley-Interscience Publication. John Wiley and Sons, Inc., New York; Hermann, Paris, 1985.
\bibitem{EL00} Euler, Leonhard.  \textit{\color{blue}De Numeris Primis Valde Magnis}, Novi Commentarii academiae scientiarum Petropolitanae 9, 1764, pp. 99-153. http://eulerarchive.maa.org/docs/originals/E283.pdf.
\bibitem{FG18} Laing Aletheia-Zomlefer, S. L. Fukshansky, S. R. Garcia. \textit{\color{blue}The Bateman-Horn Conjecture: Heuristics, History, and Applications,} http://arxiv.org/abs/1807.08899. 
\bibitem{FI10} Friedlander, John; Iwaniec, Henryk. \textit{\color{blue}Opera de cribro}. AMS Colloquium Publications, 57. American Mathematical Society, Providence, RI, 2010.
\bibitem{GM00} Granville, Andrew; Mollin, Richard A. \textit{\color{blue}Rabinowitsch revisited}. Acta Arith. 96 (2000), no. 2, 139-153. 
\bibitem{HL23} Hardy, G. H. Littlewood J.E. \textit{\color{blue}Some problems of Partitio numerorum III: On the expression of a number as a sum of primes}. Acta Math. 44 (1923), No. 1, 1-70.
\bibitem{IH78} Iwaniec, Henryk. \textit{\color{blue}Almost-primes represented by quadratic polynomials.} Invent. Math. 47 (1978), no. 2, 171-188. 
\bibitem{JW03} Jacobson, Michael J.; Williams, Hugh C. \textit{\color{blue}New quadratic polynomials with high densities of prime values}. Math. Comp. 72 (2003), no. 241, 499-519.
\bibitem{LR12} Lemke Oliver, Robert J. \textit{\color{blue}Almost-primes represented by quadratic polynomials.} Acta Arith. 151 (2012), no. 3, 241-261.
\bibitem{MA09} Matomaki, Kaisa. \textit{\color{blue}A note on primes of the form $p=aq^2+1$}. Acta Arith. 137 (2009), no. 2, 133-137.
\bibitem{MV07} Montgomery, Hugh L.; Vaughan, Robert C. \textit{\color{blue}Multiplicative number theory. I. Classical theory.} Cambridge University Press, Cambridge, 2007.
\bibitem{NW00} Narkiewicz, W. \textit{\color{blue}The development of prime number theory. From Euclid to Hardy and Littlewood}. Springer Monographs in Mathematics. Springer-Verlag, Berlin, 2000. 
\bibitem{PJ09} Pintz, Janos. \textit{\color{blue}Landau's problems on primes}. J. Theory. Nombres Bordeaux 21 (2009), no. 2, 357-404.
\bibitem{RI15} Rivin, Igor. \textit{\color{blue}Some experiments on Bateman-Horn,} http://arXiv.org/abs/1508.07821.
\bibitem{RP96} Ribenboim, Paulo. \textit{\color{blue}The new book of prime number records}, Berlin, New York: Springer-Verlag, 1996.
\bibitem{VR97} Vaughan, R. C. \textit{\color{blue}The Hardy-Littlewood method}. Second edition. Cambridge Tracts in Mathematics, 125. Cambridge University Press, Cambridge, 1997. 
\end{thebibliography}
\end{document}